\documentclass[12pt]{article}
\usepackage{amssymb}
\usepackage[mathscr]{eucal}
\usepackage{amsthm,amsmath,anysize}
\def\a{\alpha}

\def\0{\bar{0}}
\def\1{\bar{1}}

\def\t{\tilde}

\def\t{\mathfrak t}
\newtheorem{lemma}{Lemma}[section]
\newtheorem{theorem}[lemma]{Theorem}

\newtheorem{corollary}[lemma]{Corollary}

\hoffset \voffset \oddsidemargin=60pt \evensidemargin=45pt
\title{\bf On solvable elements in the Weyl algebra}

\author{Chaowen Zhang
\\ Department of
Mathematics,\\ China University
 of Mining and Technology,\\ Xuzhou, 221116, Jiang Su, P. R. China}

\date{ }

\begin{document}

\maketitle


\section{Introduction}

Let $\mathbb F$ be a field of characteristic zero. The  Weyl algebra $A$ over  $\mathbb F$ is  an associative algebra  with generators $p, q$ satisfying the relation
(\cite{dix,yt})$$[p, q]=pq-qp=1.$$ This algebra was first introduced by H. Weyl.  The $n$-th Weyl algebra $A_n$ over $\mathbb F$ is defined by $2n$ generators $p_1, q_1,\cdots, p_n,q_n, $ subject to the relations $$[p_i, p_j]=[q_i,q_j]=0,\quad [p_i, q_j]=\delta_{ij},$$ where $\delta_{ij}$ is the Kronecker symbol (see \cite{bav2,ggc,bk,yt}).\par
In \cite{dix}, Dixmier proposed the following question: Is an algebra endomorphism of $A$ necessarily an automorphism?
This question is generalized to all the Weyl algebras, known as the Dixmier Conjecture. It was shown that the Dixmier conjecture is stably equivalent to the Jacobian Conjecture (See \cite{ae, bass, bav2,bk, yt}). But currently, the Dixmier Conjecture remains open even for the case $n=1$.\par
  For convenience, we say that an element $x\in A$ is {\sl solvable} if there is an element $y\in A$ such that $$(\mathrm{ad}\ x) y=:[x, y]=1,$$   and $x$ is  {\sl unsolvable} if it is not solvable. Since \ $[y,-x]=[x, y]$ \ for $x,y\in A$, $y$ is also solvable if $x$ is so. Let $\text{Aut}(A)$ denote the group of all the algebra automorphisms of $A$. Then it is clear that if $x$ is solvable, so is $\Phi (x)$  for any $\Phi \in\text{Aut}(A)$.\
  Properties of solvable elements are studies in \cite{ggc, ggc1} and especially, a necessary condition for an element being solvable is given in \cite{ggc1}.\par
  In this paper we show that to determine all solvable elements in $A$ is   closely related to the Dixmier's  above question. We
  give some sufficient conditions for an element in $A$ being unsolvable, and we also study the properties of solvable elements.\par
The paper is arranged as follows. In Section 2 we give preliminaries. In Section 3 we study the properties of solvable elements. We show that the Dixmier's open question is equivalent to the statement that each element in $\Delta_2$ is  unsolvable. In Section 4 we prove two reduction theorems which give sufficient conditions for an element being unsolvable.  Finally, in Section 5 we give a geometric description of solvable elements.
\section{Preliminaries}
According to \cite{dix}, $A$ has a basis\ $p^iq^j,\ i,j\geq 0.$\ For each $s\in \mathbb Z$, define a subspace $A_s$ of $A$ by $$A_s=\langle p^iq^j| j-i=s\rangle.$$ It is easy to check that $A_iA_j\subseteq A_{i+j}$ for all $i, j$. Therefore $A$ becomes a $\mathbb Z$-graded algebra $A=\oplus_{i\in \mathbb Z}A_i$.\par
Denote  the element $pq\in A$ by $h$.  Then we have $$[h, p]=-p,\quad [h,q]=q.$$
A short computation shows that (See \cite[Lemma 2.1]{ggc1} or \cite[Lemma 5.1]{aj}) $$f(h)p^n=p^nf(h-n)\quad \text{and}\quad f(h)q^n=q^nf(h+n)$$ for  $f(X)\in \mathbb F[X]$ and $n\in\mathbb N$. Then each element of $A_i$ for $i>0$ (resp. $i<0$; $i=0$) can be uniquely expressed in the form $f(h)q^i$ (resp. $f(h)p^i$; $f(h)$), and hence
   each element in $A$ can be written uniquely as $$x=f_{-s}(h)p^s+\cdots +f_0(h)+\cdots +f_t(h)q^t,\quad f_i(X)\in \mathbb F[X],\ s,t\in \mathbb N.$$
Define an automorphism $\omega$ of $A$ by letting
$$\omega (p)=-q\quad\text{and}\quad \omega (q)=p.$$ Then we have \ $\omega (h)=-qp=1-h$, and hence \ $\omega (A_i)=A_{-i}$ \ for all $i$.\par
  In the following we draw the standard terminologies and notation from \cite{dix}.\par
  For $x\in A$, let $$\begin{aligned} V(x)&=\{y\in A| \text{dim}\ \langle(\mathrm{ad}\ x)^my|m=0,1,2,\cdots\rangle<\infty\},\\
   N(x)&=\{y\in A| (\mathrm{ad}\ x)^n y=0 \quad \text{for\ some}\quad n>0\},\\
                     C(x)&=\{y\in A| (\mathrm{ad}\ x)y=0\},\\
                     D(x)&=\{y\in A| (\mathrm{ad}\ x)y=\lambda y\quad \text{for\ some}\quad \lambda\in \bar{\mathbb F}\},\\
                     \mathbb F[x]&=\{f(x)\in A| f(X)\in \mathbb F[X]\},
                     \end{aligned}$$
                     where $\bar{\mathbb F}$ is an algebraic closure of $\mathbb F$.
  Then by \cite[Corollary 6.7]{dix}, the set $A-\mathbb F$ has the following partition: $$A-\mathbb F=\cup^5_{i=1}\Delta_i,$$ where
  $$\begin{aligned}
  \Delta_1&=\{x\in A-\mathbb F| N(x)=A, \ D(x)=C(x)\},\\
  \Delta_2&=\{x\in A-\mathbb F, N(x)\neq A, \ N(x)\neq C(x), \ D(x)=C(x)\},\\
  \Delta_3&=\{x\in A-\mathbb F|D(x)=A, \ N(x)=C(x)\},\\
  \Delta_4&=\{x\in A-\mathbb F|D(x)\neq A, \ D(x)\neq C(x), \ N(x)=C(x)\},\\
  \Delta_5&=\{x\in A-\mathbb F| D(x)=N(x)=C(x)\}.
  \end{aligned}$$
  It is clear that each $\Delta_i$ is invariant under every $\Phi\in \text{Aut}(A)$.
  An element $x$ is called nilpotent (resp. strictly nilpotent) if $x\in \Delta_1\cup\Delta_2$ (resp. $\Delta_1$). By definition it is easy to see that $x$ is nilpotent if and only if $N(x)\neq C(x)$.\par
  Let $f=\sum \a_{ij}X^iY^j\in \mathbb F[X,Y]$, and let $E(f)$ denote the set of pairs $(i,j)$ such that $\a_{ij}\neq 0$. For two real numbers $\rho,\sigma$ we denote $$v_{\rho,\sigma}(f)=\text{sup}_{(i,j)\in E(f)} (i\rho+j\sigma),$$  where we assign $-\infty$ to be $v_{\rho,\sigma}(0)$. We denote by $E_{\rho,\sigma}(f)$ the set of pairs $(i,j)$ in $E(f)$ such that $i\rho+j\sigma=v_{\rho,\sigma}(f)$. Then $E_{\rho,\sigma}(f)$ is nonempty if $f\neq 0$. We say that $f$ is $(\rho,\sigma)$-homogeneous of $(\rho,\sigma)$-degree $v_{\rho,\sigma}(f)$ if $E(f)=E_{\rho,\sigma}(f)$.\par
  Let $x=\sum \a_{ij}p^iq^j\in A$. Then we define $E(x), \ v_{\rho,\sigma}(x), \ E_{\rho,\sigma}(x)$ to be those for the polynomial $f=\sum \a_{ij}X^iY^j$. In particular, the polynomial $$\sum_{(i,j)\in E_{\rho,\sigma}(x)}\a_{ij} X^iY^j$$ is called the {\sl $(\rho,\sigma)$-polynomial} of $x$, and the element $$\sum_{(i,j)\in E_{\rho,\sigma}(x)}\a_{ij}p^iq^j$$ is called the {\sl $(\rho,\sigma)$-term} of $x$. \par

  \begin{lemma}(\cite[Lemma 2.4]{dix}) Let $x,y\in A$, and let $\rho, \sigma$ be real numbers such that $\rho+\sigma>0$.
Then  \par
(1) the $(\rho,\sigma)$-polynomial of $xy$ is the product of $(\rho,\sigma)$-polynomials of $x$ and $y$.\par
(2) $v_{\rho,\sigma}(xy)=v_{\rho,\sigma}(x)+v_{\rho,\sigma}(y).$
  \end{lemma}
  Let $x\in A$, and let $\rho, \sigma, k$ be three positive  numbers. Then it is clear that $$E_{k\rho,k\sigma}(x)=E_{\rho,\sigma}(x), \quad v_{k\rho,k\sigma}(x)=v_{\rho,\sigma}(x).$$
  It follows that $\rho_1/\sigma_1\neq \rho_2/\sigma_2$, if $E_{\rho_1,\sigma_1}(x)\neq E_{\rho_2,\sigma_2}(x)$, where $\rho_1,\sigma_1, \rho_2,\sigma_2$ are all positive numbers.\par
We assume in the sequel that $(\rho, \sigma)\in\mathcal P$, where $\mathcal P$ is the set of all pairs of relatively prime positive integers.
  We call $E_{\rho,\sigma}(x)$ an
 {\sl edge} of $x$ if it contains more than one point, and  a {\sl vertex} of $x$ if it is a singleton. Thus,  $E_{\rho,\sigma}(x)$ is an edge  if and only if the $(\rho,\sigma)$-polynomial of $x$ is not a monomial.  If the intersection of two edges of $x$ is a vertex, we say that the vertex joins these two edges, and the two edges are  adjacent. These terminologies can be geometrically described (see \cite[Introduction]{ggc}).  \par
 Example: Let $x=p^4 +p^3q+p^2q^2+q^3+q$. Then $x$ has two edges $$E_{1,1}(x)=\{(4,0), (3,1), (2,2)\}, \quad E_{1,2}(x)
 =\{(2,2), (0,3)\},$$  and they are joined by a vertex $E_{2,3}(x)=\{(2,2)\}$.\par
 Since $E(x)$ is finite set, $x$ can only have finitely many edges and vertices. It is also possible that $x$ has no edges. For example: $x=p^2q^2+pq+1$. Then $E_{\rho,\sigma}(x)=\{(2,2)\}$ for all $(\rho, \sigma)\in\mathcal P$.\par

  \section{Properties of solvable elements}
  Assume that $x\in A$ is solvable, and let  $y\in A$ be an element such that $[x,y]=1$. Then we have  $(\mathrm{ad}\ x)^2y=0$ and hence
  $y\in V(x)$. But $y\notin C(x)$. So we have $$V(x)\neq C(x)$$ and hence $x\notin \Delta_5$. Since $y\in N(x)$, so that $N(x)\neq C(x)$, we have $x\notin \Delta_3\cup \Delta_4$. Thus, we obtain the following corollary.

  \begin{corollary} If $x\in A$ is solvable, then $x$ is nilpotent.
  \end{corollary}
  \begin{lemma}\cite[Lemma 2.7]{dix} Assume that $\rho$ and $\sigma$ are integers such that $\rho+\sigma>0$. Let $x$ and $y$ be elements in $A$, and let $f$ and $g$ be respectively their $(\rho, \sigma)$-polynomials.   Then there exist  elements $\t,u\in A$ such that \par (a) $[x, y]=\t+u$;\par  (b) $E(\t)=E(\t, \rho, \sigma)$ and $v_{\rho, \sigma}(\t)=v_{\rho,\sigma}(x)+v_{\rho,\sigma}(y)-(\rho +\sigma)$;\par (c) $v_{\rho, \sigma}(u)<v_{\rho,\sigma}(x)+v_{\rho,\sigma}(y)-(\rho+\sigma)$.\par In addition, the following conditions are equivalent:\par
  (1) $\t=0$; \par (2) $g^{v_{\rho,\sigma}(x)}=cf^{v_{\rho,\sigma}(y)}$ for some nonzero scalar $c$.
  \end{lemma}
  \begin{lemma} Let $x\in A_i$ and  $y\in A_j$ with $i>0$ and $j<0$. If \ $[x, y]=0$, then either \ $x=0$ or\ $y=0$.
  \end{lemma}
  \begin{proof} Suppose on the contrary that $x=f_i(h)q^i\neq 0$ and $y=f_{-j}(h)p^j\neq 0$, where $f_i(X)$ and $g_{-j}(X)$ are polynomials of degree respectively $k$ and $l$.  Choose $\rho=\sigma=1$. Then we have $$v_{1,1}(x)=2k+i,\quad  v_{1,1}(y)=2l+j,$$ and the $(\rho, \sigma)$-polynomials of $x$ and $y$ are respectively $cX^kY^{k+i}$ and $c'X^{l+j}Y^l$, where $c,c'\in \mathbb C\setminus 0$.  \par Since $[x, y]=0$, we have by Lemma 3.2 that   $$(cX^kY^{k+i})^{v_{1,1}(y)}=(c'X^{l+j}Y^l)^{v_{1,1}(x)},$$  implying that $k/(k+i)=(l+j)/l$, a contradiction. Thus, we must have $x=0$ or $y=0$.
  \end{proof}
 \begin{lemma} Let $x\in A$ be  solvable,  let $y\in A$ be an element such that  $[x,y]=1$, and let $(\rho,\sigma)\in\mathcal P$ be a pair such that $v_{\rho,\sigma}(x)\geq \rho+\sigma$.  Write $[x, y]=\t+u$ as in Lemma 3.2.  Then  $\t=0$.
\end{lemma}
\begin{proof} By assumption we have  $y\in A-\mathbb F$,
   which gives $v_{\rho,\sigma}(y)>0$.\par Suppose on the contrary that $\t\neq 0$. By
 Lemma 3.2 we have  $$\begin{aligned}v_{\rho,\sigma}([x, y])&=v_{\rho,\sigma}(\t)\\&=v_{\rho,\sigma}(x)+v_{\rho,\sigma}(y)-(\rho+\sigma)\\&\geq v_{\rho,\sigma}(y)\\&>0,\end{aligned}$$  contrary to the fact that $v_{\rho,\sigma}(1)=0$. So we must have $\t=0$.
\end{proof}

  \begin{lemma}Let $x\in A_i$ with $i>0$.    If  $v_{\rho,\sigma}(x)\geq \rho +\sigma$ for some $(\rho,\sigma)\in\mathcal P$, then  $x$ is unsolvable.
  \end{lemma}
  \begin{proof} Suppose on the contrary that $x$ is solvable. Let $y\in A$ be an element such that $[x, y]=1$. Since $1\in A_0$, we must have $y\in A_{-i}$.  \par
  Let $f$ and $g$ be respectively the $(\rho,\sigma)$-polynomials of $x$ and $y$.  Write  $[x, y]=\t+u$ as in Lemma 3.2. Then we have by Lemma 3.4 that $\t=0$.  Using Lemma 3.2 once again we get $$g^{v_{\rho,\sigma}(x)}=f^{v_{\rho,\sigma}(y)},$$ which leads to a contradiction by a similar  argument as that used in the proof of Lemma 3.3. Thus, $x$ is unsolvable.
  \end{proof}

  \begin{lemma}  Let $x\in \Delta_1$ be a solvable element, and let $y\in A$ be an element such that\ $[x, y]=1$. Then $y\in \Delta_1$.
  \end{lemma}
  \begin{proof} Since $x\in \Delta_1$, by \cite[Theorem 9.1]{dix} there exists  $\Phi\in \text{Aut}(A)$ such that $\Phi(x)\in \mathbb F[q]$.  Since $\Phi$ stabilizes $\Delta_1$,  we may assume that $x\in \mathbb F[q]$.\par Let $$x=c_0+c_dq^d+c_{d+1}q^{d+1}+\cdots +c_nq^n,\quad  c_i\in\mathbb F,\  c_d\neq 0$$ and $$y=y_{-s}+\cdots +y_0+\cdots +y_t,\ y_i\in A_i,\ s,t\in \mathbb N.$$ Since $[c_iq^i, y_j]\in A_{i+j}$ for all\ $i,j$,  we have $$\begin{aligned}\begin{split}[x,y]&=[c_dq^d+\cdots +c_nq^n, y_{-s}+\cdots+y_0+\cdots  +y_t]\\&=[c_dq^d, y_{-s}]+u,\end{split}\end{aligned}$$ where $u\in \sum_{i>d-s}A_i$.\par  If $s\neq d$, since $1\in A_0$,
    we have $[c_dq^d, y_{-s}]=0$ and hence
    $y_{-s}=0$  by Lemma 3.3. It follows that $y=y_{-d}+\cdots +y_0+\cdots +y_t$.\par  If $d>1$,  then since $v_{1,1}(x)\geq d\geq 2$,  we have  by Lemma 3.5 that $x$ is unsolvable, contrary to the assumption.  Therefore, we must have $d=1$, which implies that  $[x_1, y_{-1}]=1$, so that $y_{-1}$ is solvable.   Now Lemma 3.5 implies that $v_{1,1}(y_{-1})=1$.  Thus, we have $y_{-1}=b_{-1}p$ for some $b_{-1}\in \mathbb F\setminus 0$ and hence $$y=b_{-1}p+f_0(h)+\cdots +f_t(h)q^t.$$ If $t>1$, then $[x,y]=1$ implies that $[q^n, f_t(h)q^t]=0$, which gives $$q^nf_t(h)q^t-f_t(h)q^{t+h}=
   [f_t(h-n)-f_t(h)]q^{t+n}=0.$$ Since $A$ has no zero divisors, we have $f_t(h-n)-f_t(h)=0$, implying that $f_t(X)$ is a scalar (denoted by $b_t$). It follows that $$[x, y_t]=[x, b_tq^t]=0$$ and hence $[x, y-y_t]=1$.\par
     Apply the above discussion to $y-y_t$ in place of $y$. Continue this process.  Then we obtain $$y_i=b_iq^i, \ b_i\in\mathbb F\quad \text{for\ all}\quad 0\leq i\leq t,$$   and hence $y=b_{-1}p+b_0+\cdots +b_tq^t$. \par
     By \cite[Theorem 8.10]{dix}, the group $\text{Aut}(A)$ is generated by elements $$\text{exp} (\text{ad}\ \lambda p^n),\quad  \text{exp} (\text{ad}\ \lambda q^n),\quad  \lambda \in \mathbb F,\ n\in \mathbb N\setminus 0.$$  It follows that  $\text{exp} (\text{ad}\ f(q))\in \text{Aut}(A)$ for any polynomial $f(X)\in \mathbb F[X]$ with $\text{deg}f(X)>0$.\par
       Since $p\in \Delta_1$, we have $$y=\exp\{-\text{ad}\ b_{-1}^{-1}(b_0q+\cdots +b_t\frac{1}{t+1}q^{t+1})\}(b_{-1}p)\in \Delta_1.$$
   This completes the proof.
  \end{proof}
\begin{lemma} Let $x\in A-\mathbb F$, and let $u=c_0+c_1x+\cdots +c_nx^n\in \mathbb F[x]$ with $c_n\neq 0$. If $u$ is solvable, then $n=1$.
\end{lemma}
\begin{proof} Since $u$ is solvable, we have $C(u)=\mathbb F[u]$ by \cite[Theorem 2.11]{ggc1}. Therefore $x$ is a polynomial of $u$, since $x\in C(u)$. It follows that if $n>1$, there exists a nonzero polynomial $f(X)\in \mathbb F[X]$  such that $f(x)=0$, contrary to \cite[Proposition 2.5]{dix}. So we must have $n=1$.

\end{proof}

\begin{theorem} The following statements are equivalent:\par
(1) Each endomorphism of $A$ is an automorphism.\par
(2) Each solvable element in $A$ is strictly nilpotent.\par
(3) Each element in $\Delta_2$ is unsolvable.
\end{theorem}
\begin{proof} The equivalence
of (2) and (3) is immediate from Corollary 3.1.\par
$(1)\Rightarrow (2)$:  Let $x\in A$ be solvable, and let $y\in A$ be an element such that $[x, y]=1$.
Then there exists a unique endomorphism $\varphi$ of $A$ such that $\varphi (p)=x$ and $\varphi (q)=y$. Hence $\varphi\in \text{Aut}(A)$ by (1), which shows that  $x$ is strictly nilpotent since $p$ is so. \par
$(2)\Rightarrow (1)$:  Let $\varphi$ be an endomorphism of $A$. Then $\varphi (q)$ is solvable and hence $\varphi (q)\in \Delta_1$ by (2). By \cite[Corollary 6.7 (1)]{dix}, we have $N(\varphi (q))=A$.  Then  \cite[Lemma 8.9 (1)]{dix} says that there is $\Phi\in \text{Aut} (A)$ such that $\Phi \varphi (q)\in \mathbb F[q]$. Hence we have by Lemma 3.7 that $$\Phi\varphi (q)=c_1q+c_0,\quad  c_0, c_1\in\mathbb F, \ c_1\neq 0.$$
 Since \ $[\Phi\varphi (p), \Phi\varphi (q)]=[p,q]=1$,\ the proof of Lemma 3.6 shows that $$\Phi\varphi (p)=b_{-1}p+b_0+\cdots +b_tq^t, \ b_i\in\mathbb F$$ for some $t\in\mathbb N$. Clearly we have $b_{-1}=c_1^{-1}$.\par
 Let $$\Phi'=\text{exp} \{\text{ad}\ b_{-1}^{-1}(b_0q+\cdots +b_t \frac{1}{t+1}q^{t+1})\}\in\text{Aut}(A).$$ Then we have
 $\Phi\varphi (p)=\Phi' (b_{-1}p)$ and hence $\Phi'^{-1}\Phi\varphi (p)=b_{-1}p$. Since $$\Phi'^{-1}=\text{exp} \{-\text{ad}\ b_{-1}^{-1}(b_0q+\cdots +b_t\frac{1}{t+1} q^{t+1})\},$$ we have $$\Phi'^{-1}\Phi\varphi (q)=\Phi\varphi (q)=c_1q+c_0.$$
  It is easily seen that the endomorphism $\Phi'^{-1}\Phi\varphi $ (of $A$) is indeed an automorphism. Therefore,  $\varphi \in \text{Aut}(A)$.
\end{proof}
Let us see what Lemma 3.7 implies. Recall  the  partition: $A-\mathbb F=\cup^5_{i=1}\Delta_i$. Given a polynomial $f(X)\in\mathbb F[X]$ with $\text{deg} f(X)>1$, set $$f(\Delta_i)=\{f(x)|x\in \Delta_i\}\quad\mathrm{for} \quad 1\leq i\leq 5.$$  By Dixmier's problem 6
(proved in \cite[Theorem 1.4]{aj} and \cite[p.4]{bav}), we have $f(\Delta_5)\subseteq \Delta_5$. Therefore,  $f(x)$ is unsolvable for $x\in \Delta_5$.\par
Let $x\in A-\mathbb F$. Then we have $C(x)=C(f(x))$ by \cite[Corollary 4.5]{dix}, whereas \cite[Proposition 10.3]{dix} says that $N(x)=N(f(x))$. Therefore $x$ is nilpotent (resp. strictly nilpotent) if and only if $f(x)$ is nilpotent (resp. strictly nilpotent); that is (see \cite[p.3]{bav}), $$f(x)\in \Delta_i\Leftrightarrow x\in\Delta_i,\quad \text{for}\quad i=1,2.$$
 Thus, Lemma 3.7 tells us that $f(\Delta_2)$ is a subset of $\Delta_2$ consisting of unsolvable elements. In view of Theorem 3.8, this reduces the solution to the Dixmier's open question  to determining the solvability of elements in the set
$$
\Delta_2-\underset{f,\ \mathrm{deg}f(X)>1}{\cup}f(\Delta_2).$$

\section{Reduction theorems}
 In this section we establish two reduction theorems which give sufficient conditions for an element in $A$ being
  unsolvable. \par
First,  we determine   the solvability of \ $x\in A$ \ if\ $v_{\rho,\sigma}(x)< \rho+\sigma$\ for some $(\rho,\sigma)\in\mathcal P$.
In view of $\omega\in \text{Aut}(A)$, it suffices to assume that $\rho\geq \sigma$.\par  By the Division Algorithm, there exists unique integers $l$ and $r$ such that $\rho=l\sigma +r$, and $0\leq r<\sigma$. Then $x$ is of the form $$x=c_{-1}p+c_0+c_1q+\cdots +c_lq^l,\quad c_i\in\mathbb F.$$ If $c_{-1}\neq 0$, then we have $$x=\text{exp}\{-\text{ad}\ c_1^{-1}(c_0q+\cdots +c_l \frac{1}{l+1}q^{l+1})\} (c_{-1}p)\in \Delta_1,$$ and hence $x$ is solvable since $c_{-1}p$ is so; if $c_{-1}=0$, so that $x$ is a polynomial of $q$, then the solvability of $x$ is given by  Lemma 3.7.\par
In the sequel we assume that $v_{\rho, \sigma}(x)\geq \rho+\sigma$ for all $(\rho,\sigma)\in\mathcal P$.
\subsection{The first reduction theorem}
 We give a sufficient condition for an element in $A$ being unsolvable in this subsection. The first lemma says that the majority of non-negatively graded elements are unsolvable.\par
\begin{lemma} Let $x=x_k+x_{k+1}+\cdots+x_n\in A,\ x_i\in A_i$. If $k>1$, then $x$ is unsolvable.
\end{lemma}
\begin{proof} Suppose on the contrary that $x$ is solvable. Let $$y=y_{-s}+\cdots +y_0+\cdots +y_t, \quad y_i\in A_i,\ s, t\in \mathbb N$$ be an element such that $[x, y]=1$.\par Choose $\rho=\sigma=1$. Recall
from Section 2 that each element in $A_i$ is of the form $f_i(h)q^i$ if $i\geq 0$ or $f_i(h)p^i$ if $i<0$, so that  its $(\rho,\sigma)$-degree is at least $|i|$ by Lemma 2.1.  \par
It is no loss to assume that $x_k\neq 0$. First, we claim that $y_{-s}=0$ if $s>0$. \par
Suppose that $y_{-s}\neq 0$ and $s>0$.  Write $$[x,y]=[x_k, y_{-s}]+y',\quad  y'\in \sum_{i>k-s} A_i.$$
If $s\neq k$,  then the fact that $[x_k, y_{-s}]\in A_{k-s}$ and
   $1\in A_0$ yields $[x_k, y_{-s}]=0$, contrary to Lemma 3.3. If $s=k$, then since $[x_k,y_{-k}]\in A_0$ and $y'\in \sum_{i>0} A_i$,    we have $[x_k, y_{-k}]=1$, so that $x_k$ is solvable. This contradicts Lemma 3.5 since
$$v_{\rho,\sigma}(x_k)\geq k\geq  \rho+\sigma=2.$$ Then the claim holds.  It follows that
 $$y=y_0+\cdots +y_t,$$  implying that  $[x,y]\in \sum_{i>0} A_i$, contrary to the fact that $1\in A_0$. \par
Therefore,  $x$ is unsolvable.
\end{proof}
By applying the automorphism $\omega$ of $A$,  we obtain that an element of the form $$x=x_{-n}+\cdots +x_{-k-1}+x_{-k}, \quad x_i\in A_i,\ k>1$$ is also unsolvable.\par

  Define the {\it height} of an element $x=\sum x_i, \ x_i\in A_i$ by
\ $\text{ht}(x)=\text{max}\{i|x_i\neq 0\}$.
\begin{theorem} Let $x=x_s+x_{s+1}+\cdots  +x_t\in A$, $x_i\in A_i$, $s,t\in \mathbb Z$. If $x_t\neq 0$, $t>1$, and $C(x_t)=\mathbb F[x_t]$ or, if $x_s\neq 0$, $s<-1$, and  $C(x_{-s})=\mathbb F[x_{-s}]$, then $x$ is unsolvable.
\end{theorem}
\begin{proof} In view of the automorphism $\omega$ of $A$, it suffices to prove the theorem under the assumptions $x_t\neq 0$, $t>1$, and $C(x_t)=\mathbb F[x_t]$.\par Suppose on the contrary that $x$ is solvable. Then there is $$y=y_{s'}+y_{s'+1}+\cdots+y_{t'}\in A, \quad y_j\in A_j,\ s',t'\in\mathbb Z$$ such that $[x, y]=1$. Now that $y$
is solvable,  we have  $\text{ht}(y)=t'\geq -1$ by the conclusion following Lemma 4.1.\par
 Write $$[x,y]=[x_t, y_{t'}]+u, \quad u\in \sum_{i<t+t'}A_i. $$
 Since $t+t'\geq 1$ and $1\in A_0$,     we obtain $[x_t, y_{t'}]=0$ and hence
$y_{t'}\in C(x_t)=\mathbb F[x_t]$. Let $f(X)\in \mathbb F[X]$ be the polynomial such that $y_{t'}=f(x_t)$. Since $x_t^k\in A_{tk}$ for all $k\in \mathbb N$, we have $$f(X)=cX^n\quad \mathrm{for\ some}\quad c\in \mathbb F\setminus 0\ \ \mathrm{and}\ \ n\in\mathbb N$$  and hence $y_{t'}=cx_t^n$, implying that $\text{ht}(y)=t'=nt\geq 0$.\par
If $t'=0$, so that $n=0$ since $t>1$, then $y_0=c\in\mathbb F$. It follows that  $$[x, y_{s'}+y_{s'+1}+\cdots +y_{-1}]=[x, y-y_0]=1,$$ a contradiction.\par
If $t'>0$, since
   $$\begin{aligned} (y-f(x))_{t'}&=y_{t'}-f(x)_{t'}\\ &=y_{t'}-cx^n_t\\ &=0,\end{aligned}$$ we have $\text{ht}(y-f(x))<t'=\text{ht}(y)$.  Since $[x, y-f(x)]=1$, the above discussion applies to $y-f(x)$ in place of  $y$ as long as $\text{ht} (y-f(x))>0$.  Continue  this process.  Then we will obtain an element $y'$ such that $[x,y']=1$ and $\text{ht} (y')\leq 0$, a contradiction.   \par  Thus,  $x$ must be unsolvable.
\end{proof}

\subsection{ The second reduction theorem}

In this subsection we establish the second sufficient condition for $x$ being unsolvable.

 \begin{lemma}(\cite[Proposition 7.4]{dix}) Let $\rho,\sigma$ be positive integers such that $\rho\nmid\sigma$ and
  $\sigma\nmid\rho$. Let $x\in A$ with $v_{\rho,\sigma}(x)>\rho +\sigma$. If the $(\rho,\sigma)$-polynomial of $x$ is not a monomial, then  $x\in \Delta_5$.
 \end{lemma}
 By the lemma, if an element $x\in A$ is solvable, so that $x\notin \Delta_5$,  and if $x$ has an edge $E_{\rho,\sigma}(x)$ with $(\rho,\sigma)\in\mathcal P$, then the assumption $v_{\rho,\sigma}(x)>\rho +\sigma$ yields either $\rho|\sigma$ or $\sigma|\rho$; that is, either $\rho=1$ or $\sigma=1$.\par
   In view of the automorphism $\omega$ of $A$, we assume that $\rho=n\geq 1$ and $\sigma=1$ in the remainder of this subsection.\par
\begin{lemma} Let $f, g\in \mathbb F[X, Y]$ be two $(\rho,\sigma)$-homogeneous
polynomials with $(\rho,\sigma)$-degrees respectively  $w$ and $v$. Assume that $f$ is not equal to $f_1^m$ ($m>1$) for any polynomial $f_1\in \mathbb F[X,Y]$.  If $f^v=g^w$, then $w|v$ and hence $g=f^{v/w}$.
\end{lemma}
\begin{proof} Since $\rho=n$ and $\sigma=1$,  the $(\rho,\sigma)$-degree of a $(\rho, \sigma)$-homogeneous polynomial is divisible by $n$. Assume that $w=w_1n$ and $v=v_1n$ for some integers $v_1$ and  $w_1$. Then we get $$\begin{aligned} f&=cX^{w_1}(Y^n/X\ -\lambda_1)^{s_1}\cdots (Y^n/X\ -\lambda_k)^{s_k}, \\ g&=c'X^{v_1}(Y^n/X\ -\mu_1)^{r_1}\cdots (Y^n/X\ -\mu_l)^{r_l}\end{aligned}$$ with $s_1, \dots, s_k$,
$r_1,\dots r_l\in \mathbb N\setminus 0$ and $c, c'\in \mathbb F\setminus 0$, where $\lambda_1, \dots, \lambda_k$ (resp. $\mu_1,\dots, \mu_l$) are  distinct numbers in $\bar{\mathbb F}$.\par Since $f^v=g^w$,\ we have $k=l$ and, by an rearrangement of indices,  $\lambda_i=\mu_i$ for all $i$. It follows that  $s_iv=r_iw$ for $1\leq i\leq k$.\par Since  $f$ is not equal to $f_1^m$ ($m>1$) for any polynomial $f_1$, we get $(s_1, \dots, s_k)=1$. Then since $w$ divides $s_iv$ for all $i$, \ $w$ divides $$(s_1v,\dots,s_kv)=(s_1,\dots, s_k)v=v,$$ implying that $g=f^{v/w}$, as desired.
\end{proof}

\begin{theorem}Let $x\in A$ be an element such that $v_{\rho, \sigma}(x)\geq \rho+\sigma=n+1$, and  let $f$ be its $(\rho,\sigma)$-polynomial. If $f$ is not equal to $f_1^m$ ($m>1$) for any $f_1\in\mathbb F[X,Y]$,  then $x$ is unsolvable.
\end{theorem}
\begin{proof}Suppose on the contrary that $x$ is solvable. Let $y\in A$ be an element such that $[x,y]=1$, and  let $g$ be its $(\rho,\sigma)$-polynomial.  \par Write $[x,y]=\t+u$ as in  Lemma 3.2. Then we have $\t=0$ by Lemma 3.4. Hence, Lemma 3.2 says that  $$g^{v_{\rho,\sigma}(x)}=cf^{v_{\rho,\sigma}(y)} \quad \text{for \ some}\ c\neq 0,$$ so that  $g=c_lf^l$ ($c_l\neq 0$) by Lemma 4.4,  where $l=v_{\rho,\sigma}(y)/v_{\rho,\sigma}(x)\in \mathbb N\setminus 0$.\par  Note that $$[x, y-c_lx^l]=[x,y]=1.$$ But $v_{\rho,\sigma}(y-c_lx^l)<v_{\rho,\sigma}(y)$, since $v_{\rho,\sigma}(y)-v_{\rho,\sigma}(c_lx^l)=g-c_lf^l=0$. \par Apply the above discussion to $y-c_lx^l$ in place of $y$,
if $y-c_lx^l$ is not a scalar. Continue the process.
 Since $\rho>0$ and $\sigma>0$, so that the $(\rho, \sigma)$-degree of every element in $A-\mathbb F$ is positive, ultimately we will obtain a polynomial $h(X)\in \mathbb F[X]$ such that $y-h(x)$ is a scalar. It follows that $[x, y-h(x)]=0$, contrary to the fact that
$[x,y]=[x,y-h(x)]$.\par Therefore, $x$ must be unsolvable.
\end{proof}

\section{A geometric characterization of solvable elements}
In this section we give a necessary condition for an element $x\in A$ being solvable and having more than one edges.
 \begin{lemma} Let \ $i, j, i_0,j_0, a, b$ be real numbers such that $a<b$. If \ $ia+j\leq i_0a+j_0$ \ and \ $ib+j<i_0b+j$, then \ $iX+j<i_0X+j_0$\ for all \ $X\in (a,b)$.
\end{lemma}
\begin{proof} Define a function $$f(X)=(i-i_0)X+(j-j_0), \quad X\in \mathbb R.$$ Then the graph of $Y=f(X)$ is a line,
 on which the segment with $X$-coordinates between $a$ and $b$ lies below the $X$-axis,  since
$f(a)\leq 0$ and $f(b)<0$ by assumption.  Thus, we have $f(X)<0$ and hence  $iX+j<i_0X+j_0$ \ for all $X\in (a,b)$.  .
\end{proof}
Let
$x=\sum\a_{ij}p^iq^j\in A$, and let $E_{\rho_1,\sigma_1}(x)$ and $E_{\rho_2,\sigma_2}(x)$,\ $(\rho_1,\sigma_1), (\rho_2,\sigma_2)\in \mathcal P$, \ be two distinct
 edges of $x$.
  Then we have the following lemma.
 \begin{lemma} Then intersection $E_{\rho_1,\sigma_1}(x)\cap E_{\rho_2,\sigma_2}(x)$ contains at most one pair of integers.
 \end{lemma}
 \begin{proof} Recall from Section 2 that $\rho_1/\sigma_1\neq \rho_2/\sigma_2$. If there are two pair of integers $(i_0, j_0), (i, j)$ contained in $E_{\rho_1,\sigma_1}(x)\cap E_{\rho_2,\sigma_2}(x)$, then we have
 $$\begin{aligned} i\rho_1+j\sigma_1&=i_0\rho_1+j_0\sigma_1,\\  i\rho_2+j\sigma_2&=i_0\rho_2+j_0\sigma_2,\end{aligned}$$ so that $$\begin{aligned}(i-i_0)\rho_1+(j-j_0)\sigma_1&=0, \\(i-i_0)\rho_2+(j-j_0)\sigma_2&=0.\end{aligned}$$ The assumption $\rho/\sigma_1\neq \rho_2/\sigma_2$ implies that $i-i_0=j-j_0=0$, and hence $(i,j)=(i_0,j_0)$. This completes the proof.
 \end{proof}

Let
$x=\sum\a_{ij}p^iq^j\in A$,
  and assume that $$E_{\rho_1,\sigma_1}(x)\quad \mathrm{and}\quad  E_{\rho_2,\sigma_2}(x), \quad (\rho_1,\sigma_1), \ (\rho_2,\sigma_2)\in\mathcal P,$$  are two adjacent
 edges of $x$ joined  by a vertex $(i_0,j_0)$. By definition, the  $(\rho_1,\sigma_1)$-term of $x$ is of the form
 $$\cdots +\a_{i_0,j_0}p^{i_0}q^{j_0},$$ and the $(\rho_2,\sigma_2)$-term of $x$ is of the form $$\a_{i_0,j_0}p^{i_0}q^{j_0}+\cdots .$$

 \begin{lemma}(see \cite[Proposition 3.7(3)]{ggc})  There exists $(\rho, \sigma)\in\mathcal P$  such that the $(\rho,\sigma)$-term of $x$ is $\a_{i_0,j_0}p^{i_0}q^{j_0}.$
 \end{lemma}
 \begin{proof} Without loss of generality we assume that $\rho_1/\sigma_1< \rho_2/\sigma_2$.  Let $(\rho, \sigma)\in \mathcal P$ be pair such that $\rho/\sigma\in (\rho_1/\sigma_1, \rho_2/\sigma_2)$.  We show that $(\rho,\sigma)$ is a desired pair.\par
 By definition we need to show that $$i\rho+j\sigma<i_0\rho+j_0\sigma\quad \text{for\ all}\quad  (i,j)\in E(x)\setminus(i_0,j_0).$$
  Let $(i,j)\in E(x)\setminus (i_0,j_0)$.
Then since $(i_0,j_0)$ joins $E_{\rho_1,\sigma_1}(x)$ and $E_{\rho_2,\sigma_2}(x)$, meaning that $i_0\rho_l+j_0\sigma_l=v_{\rho_l,\sigma_l}(x)$ for $l=1,2$, we  have by Lemma 5.2 $$\text{either} \quad i\rho_1+j\sigma_1<i_0\rho_1+j_0\sigma_1\quad  \text{or}\quad i\rho_2+j\sigma_2<i_0\rho_1+j_0\sigma_1.$$ Without loss of generality we assume that $$i\rho_1+j\sigma_1\leq i_0\rho_1+j_0\sigma_1,\qquad i\rho_2+j\sigma_2<i_0\rho_2+j_0\sigma_2$$ or, equivalently,
 $$i\rho_1/\sigma_1+j\leq i_0\rho_1/\sigma_1+j_0,\qquad i\rho_2/\sigma_2+j< i_0\rho_2/\sigma_2+j_0.$$  Then Lemma 5.1 says that $i\rho/\sigma+j<i_0\rho/\sigma+j_0$ and hence
$i\rho+j\sigma<i_0\rho+j_0\sigma$. Therefore, the $(\rho,\sigma)$-term of $x$ is $\a_{ij}p^{i_0}q^{j_0}$.\par
 \end{proof}
 Note:  The above proof shows that for any  pair $(\rho, \sigma)\in\mathcal P$ such that $$\rho/\sigma\in (\rho_1/\sigma_1, \rho_2/\sigma_2),$$
   the $(\rho,\sigma)$-term of $x$ is the same.\par

 \begin{lemma} Keep the assumptions before Lemma 5.3 on $x$.  In addition, assume  that $v_{\rho,\sigma}(x)\geq \rho+\sigma$ for all  $(\rho,\sigma)\in \mathcal P$. If $x$ is solvable with $y\in A$ such that $[x,y]=1$,   then $y$ also has adjacent edges $E_{\rho_1,\sigma_1}(y)$, $E_{\rho_2,\sigma_2}(y)$.
\end{lemma}
\begin{proof} By our assumptions,   the $(\rho_1,\sigma_1)$-polynomial and the $(\rho_2,\sigma_2)$-polynomial of $x$  are respectively  $$\cdots +\a_{i_0,j_0}X^{i_0}Y^{j_0}$$ and $$\a_{i_0,j_0}X^{i_0}Y^{j_0}+\cdots .$$ \par Since $x$ is solvable, by Theorem 4.5 we may write these two
polynomials respectively as $f_1^{r_1}$ ($r_1>1$) and $f_2^{r_2}$ ($r_2>1$) for some $f_1, f_2\in \mathbb F[X,Y]$. We may further assume  that, for $l=1,2$,  $f_l$  is no longer equal to $\kappa_l^k$ ($k>1$) for any $\kappa_l\in \mathbb F[X,Y]$.
 Since $E_{\rho_1,\sigma_1}(x)$ and $E_{\rho_2,\sigma_2}(x)$ are edges of $x$, neither $f_1$ nor $f_2$ is a monomial. \par For each pair $(\rho_l, \sigma_l)$, $l=1,2$, write $[x,y]=\t+u$ as in Lemma 3.2. Then  we have $\t=0$ by Lemma 3.4. Hence, by  Lemma 3.2 and Lemma 4.4 the $(\rho_l,\sigma_l)$-polynomial of $y$ is of the form $$c_lf_l^{R_l}, \quad  R_l\in \mathbb N\setminus 0,\ c_l\in \mathbb F\setminus 0.$$ Thus, $y$ has the edges  $E_{\rho_1,\sigma_1}(y)$ and $E_{\rho_2,\sigma_2}(y)$, which we are to show are adjacent.\par Let $(\rho, \sigma)\in\mathcal P$ be a pair of integers such that  $\rho/\sigma\in (\rho_1/\sigma_1, \rho_2/\sigma_2)$. By the proof of Lemma 5.3, the $(\rho,\sigma)$-polynomial of $x$ is $\a_{i_0,j_0}X^{i_0}Y^{j_0}$.   Let $y=\sum \beta_{ij}p^iq^j$ and let $\tau$ be its $(\rho,\sigma)$-polynomial. Write $[x,y]=\t+u$ as in Lemma 3.2. Then we have $\t=0$ by Lemma 3.4.
  Hence Lemma 3.2 says that $$\tau^{v_{\rho,\sigma}(x)}=c(\a_{i_0,j_0}X^{i_0}Y^{j_0})^{v_{\rho,\sigma}(y)} \quad \text{for some}\ c\in\mathbb F\setminus 0.$$  Therefore $\tau$  is  a monomial, which we denote  by $\beta_{st} X^sY^t$.  Then  $$(X^sY^t)^{v_{\rho,\sigma}(x)}=(X^{i_0}Y^{j_0})^{v_{\rho,\sigma}(y)},$$ implying that $s/t=i_0/j_0$, so that $(s, t)=(Ki_0,Kj_0)$ for some $K>0$. Since $\tau$ is the $(\rho,\sigma)$-polynomial of $y$, we have $$v_{\rho,\sigma}(y)=s\rho+t\sigma=Ki_0\rho+Kj_0\sigma,$$ which    implies  that
  $$K=\text{max}\{k>0|(ki_0,kj_0)\in E(y)\},$$ so that $(s,t)$ is independent of the choice of $(\rho, \sigma)$ as above.    In other words, the element $\beta_{st}p^sq^t$
is the $(\rho,\sigma)$-term of $y$  for any  $(\rho,\sigma)\in \mathcal P$ such that $\rho/\sigma\in (\rho_1/\sigma_1, \rho_2/\sigma_2)$.\par
 We now show that the vertex $(s,t)$ joins $E_{\rho_1,\sigma_1}(y)$ and $E_{\rho_2,\sigma_2}(y)$.\par Choose a sequence of pairs  $(\rho_{(n)},\sigma_{(n)})\in\mathcal P,\ n=1,2,\dots$, such that  $$\rho_{(n)}/\sigma_{(n)}\in (\rho_1/\sigma_1, \rho_2/\sigma_2)\quad \text{and}\quad \underset{n\rightarrow \infty}{\text{lim}}\ \rho_{(n)}/\sigma_{(n)}=\rho_1/\sigma_1.$$ For each $n$, since $\beta_{st}p^sq^t$ is the $(\rho_{(n)}, \sigma_{(n)})$-term of $y$, we have $$i\rho_{(n)}+j\sigma_{(n)}<s\rho_{(n)}+t\sigma_{(n)}\quad \text{for all}\quad  (i,j)\in E(y)\setminus (s,t)$$ and hence $$i\rho_{(n)}/\sigma_{(n)}+j<s\rho_{(n)}/\sigma_{(n)}+t\quad \text{for all}\quad  (i,j)\in E(y)\setminus (s,t).$$ Taking the limit as $n\rightarrow \infty$, we get $$i\rho_1/\sigma_1+j\leq s \rho_1/\sigma_1+t \quad \text{for all}\quad  (i,j)\in E(y)\setminus (s,t),$$ so that $$i\rho_1+j\sigma_1\leq s\rho_1+t\sigma_1 \quad \text{for all}\quad  (i,j)\in E(y)\setminus (s,t),$$ implying that $(s,t)\in E_{\rho_1,\sigma_1}(y)$.\par Similarly we obtain $(s,t)\in E_{\rho_2,\sigma_2}(y)$.
  Then Lemma 5.2 says that   $$E_{\rho_1,\sigma_1}(y)\cap E_{\rho_2,\sigma_2}(y)=\{(s,t)\},$$ so that the edges
  $E_{\rho_1,\sigma_1}(y)$ and  $E_{\rho_2,\sigma_2}(y)$ are adjacent.
\end{proof}

  Let $x\in A$ be an element having edges $$E_{\rho_1,\sigma_1}(y), \dots, E_{\rho_k,\sigma_k}(y), \quad (\rho_1,\sigma_1), \dots, (\rho_k, \sigma_k)\in \mathcal P, \ k\geq 2.$$   Then $\rho_1/\sigma_1,\dots, \rho_k/\sigma_k$ are $k$ distinct rational numbers. It's no loss of generality to assume that $$\rho_1/\sigma_1<\rho_2/\sigma_2<\cdots <\rho_k/\sigma_k.$$ For each $l$, $1\leq l<k$,
take $(\rho,\sigma)\in\mathcal P$ such that $\rho/\sigma\in (\rho_l/\sigma_l,\  \rho_{l+1}/\sigma_{l+1})$.\   Then $E_{\rho,\sigma}(x)$ is no longer an edge, hence it must be a vertex which, by a similar argument as in the   proof of Lemma 5.4,  joins $E_{\rho_l,\sigma_l}(x)$ and $E_{\rho_{l+1},\sigma_{l+1}}(x)$. Therefore,
      the edges $E_{\rho_l,\sigma_l}(x)$ and $E_{\rho_{l+1},\sigma_{l+1}}(x)$ are adjacent.\par We claim that $E_{\rho_l,\sigma_l}(x)$ and $E_{\rho_{l+j},\sigma_{l+j}}(x)$ are not adjacent if $j>1$. Suppose on the contrary that they are adjacent. Then by the proof of Lemma 5.3, $E_{\rho,\sigma}(x)$ is a vertex for any $(\rho,\sigma)\in\mathcal P$ such that $\rho/\sigma\in (\rho_l/\sigma_l, \rho_{l+j}/\sigma_{l+j})$,  contrary to the fact that $E_{\rho_{l+1},\sigma_{l+1}}(x)$ is an edge.   Therefore, two distinct edges $$E_{\rho_i,\sigma_i}(x),\ E_{\rho_j,\sigma_j}(x),\quad 1\leq i<j\leq k,$$  are adjacent if and only if $i=j-1$.\par
We are now ready to study solvable elements in $A$ using these conclusions. \par In the following, assume that $x\in A$
is  a solvable element having edges $$E_{\rho_1,\sigma_1}(x), \dots, E_{\rho_k,\sigma_k}(x),\quad (\rho_1,\sigma_1),\dots ,(\rho_k,\sigma_k)\in\mathcal P, \ k\geq 2,$$ where $\rho_1/\sigma_1<\cdots <\rho_k/\sigma_k$.
We also assume that $$v_{\rho,\sigma}(x)\geq \rho+\sigma\quad \text{for all}\quad (\rho,\sigma)\in\mathcal P.$$
Then by Theorem 4.5, for each $l$,  $ 1\leq l\leq k$, the $(\rho_l,\sigma_l)$-polynomial of $x$ is of the form $f_l^{r_l}$ ($r_l>1$) for some $f_l\in \mathbb F[X,Y]$. We assume that each $f_l$ is no longer equal to $\kappa^n$ ($n>1$) for any polynomial $\kappa\in\mathbb F[X,Y]$.
\begin{lemma} With these assumptions on $x$, we have   $$(r_l, r_{l+1})>1\quad \text{for all}\quad 1\leq l<k.$$
\end{lemma}
\begin{proof} It suffices to show that $(r_1, r_2)>1$. The other inequalities can be proved similarly.\par
We assume that $(r_1, r_2)=1$ and derive a contradiction.\par
 Let $x=\sum \a_{ij} p^iq^j$, and let $(i_0, j_0)$ be the vertex of $x$ joining $E_{\rho_1,\sigma_1}(x)$ and $E_{\rho_2,\sigma_2}(x)$. Then we have $$f_1^{r_1}=\cdots +\a_{i_0,j_0}X^{i_0}Y^{j_0}$$ and $$f_2^{r_2}=\a_{i_0,j_0}X^{i_0}Y^{j_0}+\cdots .$$
 Let $y\in A$ be an element such that $[x, y]=1$. By Lemma 5.4, $y$ also has adjacent edges $E_{\rho_l,\sigma_l}(y)$, $l=1,2$. In particular,
 the proof of Lemma 5.4 shows that, for $l=1,2$,  the $(\rho_l,\sigma_l)$-polynomial  of $y$ is also of the form $c_lf_l^{R_l}$ for some  $R_l\in \mathbb N\setminus 0$ and $c_l\in\mathbb F\setminus 0$.\par
 Let $(s,t)$ be the vertex
joining $E_{\rho_1,\sigma_1}(y)$ and $E_{\rho_2,\sigma_2}(y)$.
  Then  we have $$c_1f_1^{R_1}=\cdots +\beta_{st}X^sY^t$$ and $$c_2f_2^{R_2}=\beta_{st}X^sY^t+\cdots .$$ From the proof of Lemma 5.4 we also obtain $i_0/j_0=s/t$. \par For brevity, set $$\theta=:s/i_0\ (=t/j_0),\quad v_1=:v_{\rho_1,\sigma_1}(f_1),\quad v_2=:v_{\rho_2,\sigma_2}(f_2).$$ Since both $x$ and $y$ are solvable, so that $x,y\in A-\mathbb F$, \  $v_1$ and $v_2$ are both nonzero.\par
 From above we have $$\begin{aligned}v_{\rho_1,\sigma_1}(x)&=i_0\rho_1+j_0\sigma_1=r_1v_1,\quad &&v_{\rho_2,\sigma_2}(x)=i_0\rho_2+j_0\sigma_2=r_2v_2,\\
v_{\rho_1,\sigma_1}(y)&=s\rho_1+t\sigma_1=R_1v_1,\quad &&v_{\rho_2,\sigma_2}(y)=s\rho_2+t\sigma_2=R_2v_2,\end{aligned}$$ implying that $$R_1v_1=\theta r_1v_1, \quad
R_2v_2=\theta r_2v_2$$ and hence $$R_1=\theta r_1,\quad R_2=\theta r_2.$$ Write the rational number $\theta$ as $m/n$ such that $(m,n)\in\mathcal P$.  Then we get $$nR_1=mr_1,\quad  nR_2=mr_2$$  and hence \ $n|r_1, \ n|r_2$, \ since $(n,m)=1$.  It follows that  $n=1$, which yields $$R_1=mr_1,\quad R_2=mr_2.$$
Set $y'=y-c_1x^m$. Then since $c_1f_1^{R_1}-c_1(f_1^{r_1})^m=0$,  the $(\rho_1,\sigma_1)$-degree of the $(\rho_1,\sigma_1)$-polynomial of $y'$ is strictly less than that of $y$; that is, $v_{\rho_1,\sigma_1}(y')<v_{\rho_1,\sigma_1}(y)$.
  Since $[x,y']=1$,  the above discussion applies to $y'$ in place of $y$ unless $y'$ is a scalar. Continue this process. Then we will obtain that $y$ is a polynomial of $x$, so that $[x,y]=0$, a contradiction. \par Therefore, we must have $(r_1, r_2)>1$, as desired.
\end{proof}

\begin{theorem} Keep the  assumptions on $x$ before Lemma 5.5. Then $$(r_1,r_2,\dots, r_k)>1.$$
\end{theorem}
\begin{proof} The case $k=2$ is immediate from Lemma 5.5. Thus we assume that $k>2$.\par
We assume that $(r_1, r_2,\dots, r_k)=1$ and derive a contradiction.\par
Let $y\in A$ be an element such that $[x,y]=1$. Then Lemma 5.4 says that
$y$ also has $k$ edges $$E_{\rho_1,\sigma_1}(y),\dots,  E_{\rho_k,\sigma_k}(y)$$ such that, for each $l$, $1\leq l<k$, the edges $E_{\rho_l,\sigma_l}(y)$ and $E_{\rho_{l+1},\sigma_{l+1}}(y)$ are joined by a vertex which we denote now by
$(s_l,t_l)$.  In addition,  from the proof of Lemma 5.5 we have that, for $1\leq l\leq k$,  the $(\rho_l,\sigma_l)$-polynomial of $y$ equals \ $c_lf_l^{R_l}$\ for some $R_l\in\mathbb N\setminus 0$ and $c_l\in\mathbb F\setminus 0$. \par  Set  $v_l=:v_{\rho_l,\sigma_l}(f_l)$ for  $1\leq l\leq k$. In view of the proof of Lemma 5.5, we have $v_l\neq 0$ for all $l$.\par Since for all $l$ with  $1\leq l <k$,  $(i_l, j_l)$ joins the edges $E_{\rho_l,\sigma_l}(x)$, $E_{\rho_{l+1},\sigma_{l+1}}(x)$,  and  $(s_l, t_l)$ joins the edges $E_{\rho_l,\sigma_l}(y)$, $E_{\rho_{l+1},\sigma_{l+1}}(y)$,    we have
$$ (1)\qquad \begin{aligned}v_{\rho_1,\sigma_1}(x)&= i_1\rho_1+j_1\sigma_1=r_1v_1, \\
v_{\rho_2,\sigma_2}(x)&=i_1\rho_2+j_1\sigma_2=r_2v_2,\\
v_{\rho_2,\sigma_2}(x)&=i_2\rho_2+j_2\sigma_2=r_2v_2,\\
    &                          \dots,\\
v_{\rho_{k-1},\sigma_{k-1}}(x)&= i_{k-1}\rho_{k-1}+j_{k-1}\sigma_{k-1}=r_{k-1}v_{k-1},\\ v_{\rho_k,\sigma_k}(x)&=i_{k-1}\rho_k+j_{k-1}\sigma_k=r_kv_k,\end{aligned}$$
 and
$$(2)\qquad \begin{aligned} v_{\rho_1,\sigma_1}(y)&=s_1\rho_1+t_1\sigma_1=R_1v_1, \\
v_{\rho_2,\sigma_2}(y)&=s_1\rho_2+t_1\sigma_2=R_2v_1,\\
v_{\rho_2,\sigma_2}(y)&= s_2\rho_2+t_2\sigma_2=R_2v_2,\\
&        \dots,\\
  v_{\rho_{k-1},\sigma_{k-1}}(y)&=s_{k-1}\rho_{k-1}+t_{k-1}\sigma_{k-1}=R_{k-1}v_{k-1},\\ v_{\rho_k,\sigma_k}(y)&=s_{k-1}\rho_k+t_{k-1}\sigma_k=R_kv_k.\end{aligned}$$

From the proof of Lemma 5.5 we have \ $s_l/i_l=t_l/j_l $ \ for $1\leq l< k$.
  Then by the identities $i_1\rho_2+ j_1\sigma_2=i_2\rho_2+j_2\sigma_2$ obtained from (1) and $ s_1\rho_2+ t_1\sigma_2=s_2\rho_2+t_2\sigma_2$ obtained from  (2),
 we have  $$s_1/i_1=t_1/j_1=s_2/i_2=t_2/j_2.$$ Inductively  we obtain $$
s_1/i_1=t_1/j_1=s_2/i_2=t_2/j_2=\cdots =s_{k-1}/i_{k-1}=t_{k-1}/j_{k-1}.$$
 Denote this rational number by $m/n$ such that $(m,n)\in \mathcal P$. By comparing the identities (1) with the identities (2)
we obtain $$R_l/r_l=m/n \quad \mathrm{for}\quad 1\leq l\leq k.$$ Then the assumption $(r_1,\dots,r_k)=1$ leads to a contradiction by a similar argument as that used in the  proof of Lemma 5.5. Therefore, we must have $(r_1,\dots,r_k)>1,$ as desired.
\end{proof}

\def\refname{
\end{document}
\centerline{\bf REFERENCES}}
\begin{thebibliography}{10}

\bibitem {ae}  Kossivi Adjamagbo, Arno van den Essen, \textit {\rm A proof of the equivalence of the Dixmier, Jacobian and Poisson conjectures},
Acta Math. Vietnam  \textbf{32(2-3)} (2007), 205-214.
\bibitem {bass}  H. Bass, E. Connell, D. Wright,  \textit {\rm The Jacobian conjecture: reduction of degree and formal expansion of the inverse},
Bull. Amer. Math. Soc.  \textbf{7(2)} (1982), 287-330.
\bibitem{bav} V. V. Bavula, Dixmier's problem 6 for the Weyl algebra (the generic type problem),
arXiv:math/0402244v1. 2004.
\bibitem {bav2}  V. V. Bavula, \textit {\rm A question of Rentschler and the Dixmier Problem},
Ann. Math.  \textbf{154} (2001), 683-702.
\bibitem {dix}  J. Dixmier, \textit {\rm Sur Les Algebres de Weyl},
Bull. Soc. Math. France  \textbf{96} (1968), 209-242.
\bibitem {ggc}  J. A. Guccione, J. J. Guccione, C. Valqui, \textit {\rm The Dixmier Conjecture and the shape
of possible counterexamples},
J. Algebra  \textbf{399} (2014), 581-633.
\bibitem {ggc1}  J. A. Guccione, J. J. Guccione, C. Valqui, \textit {\rm On the centralizers in the Weyl algebra},
Proc. Amer. Math. Soc.   \textbf{140(4)} (2012), 1233-1241.
\bibitem {aj}  A. Joseph, \textit {\rm The Weyl algebra-semisimple and nilpotent elements},
Amer. J. Math.  \textbf{97} (1975), 597-615.
\bibitem {bk}  Alexei Belov-Kanel and Maxim Kontsevich, \textit {\rm The Jacobian conjecture is stably equivalent
to the Dixmier conjecture},
Moscow Math. Journal  \textbf{7} (2007), 209-218.
\bibitem {ml}  L. Makar-Limanov, \textit {\rm On automorphisms of Weyl algebra},
Bull. Soc. Math. France  \textbf{112} (1984), 359-363.
\bibitem{yt} Y. Tsuchimoto, Endomorphisms of Weyl algebra and $p$-curvatures,
Osaka J. Math. \textbf{42} (2005), 435-452.
\end{thebibliography}
